\documentclass[12pt]{article}
\usepackage{amsmath,amsthm,amsfonts,amssymb}
\usepackage[pdftex,pdfborder={0 0 0}]{hyperref}
\usepackage{fullpage}
\usepackage{multirow}
\usepackage{array}
\usepackage{bbm}
\usepackage[noblocks]{authblk}
\usepackage{verbatim}
\usepackage{tikz}
\usepackage{float}
\usepackage{enumerate}
\usepackage{diagbox}
\usepackage{comment}
\usetikzlibrary{arrows}

\newtheorem{theorem}{Theorem}[section]
\newtheorem{lemma}[theorem]{Lemma}

\newtheorem{corollary}[theorem]{Corollary}
\newtheorem{conjecture}[theorem]{Conjecture}

\theoremstyle{definition}
\newtheorem{definition}[theorem]{Definition}

\newtheorem{question}[theorem]{Question}

\newlength{\Oldarrayrulewidth}

\newcommand{\N}{\mathbb{N}}

\renewcommand{\mod}[2]{\equiv#1\textup{ (mod }#2\textup{)}}

\def\m@th{\mathsurround=0pt}
\def\sm#1{\null\,\vcenter{\baselineskip9pt\lineskip.23ex\m@th
    \ialign{\hfil$\scriptstyle##$\hfil&&\ \hfil$\scriptstyle##$\hfil\crcr
    \mathstrut\crcr\noalign{\kern-\baselineskip}
    #1\crcr\mathstrut\crcr\noalign{\kern-\baselineskip}}}\,}
\def\smnp#1{\null\,\vcenter{\baselineskip9pt\lineskip.23ex\m@th
    \ialign{\hfil$\scriptstyle##$\hfil&&\ \ \hfil$\scriptstyle##$\hfil\crcr
    \mathstrut\crcr\noalign{\kern-\baselineskip}
    #1\crcr\mathstrut\crcr\noalign{\kern-\baselineskip}}}\,}

\begin{document}

\title{On binomial coefficients associated with Sierpi\'{n}ski and Riesel numbers}

\author[1]{Ashley~Armbruster\thanks{aarmbruster0@frostburg.edu}}

\author[2]{Grace~Barger\thanks{grace.barger@salem.edu}}

\author[3]{Sofya~Bykova\thanks{sab444@cornell.edu}}

\author[4]{Tyler~Dvorachek\thanks{dvortj19@uwgb.edu}}

\author[5]{Emily~Eckard\thanks{e.m.eckard@email.msmary.edu}}

\author[6]{Joshua~Harrington\thanks{joshua.harrington@cedarcrest.edu}}

\author[7]{Yewen~Sun\thanks{yewen@ucsb.edu}}

\author[8]{Tony~W.~H.~Wong\thanks{wong@kutztown.edu}}

\affil[1]{Department of Mathematics, Frostburg State University}
\affil[2]{Department of Mathematics, Salem College}
\affil[3]{Department of Mathematics, Cornell University}
\affil[4]{Department of Mathematics, University of Wisconsin-Green Bay}
\affil[5]{Department of Mathematics, Mount St.\ Mary's University}
\affil[6]{Department of Mathematics, Cedar Crest College}
\affil[7]{Department of Mathematics, University of California, Santa Barbara}
\affil[8]{Department of Mathematics, Kutztown University of Pennsylvania}
\date{\today}

\maketitle

\begin{abstract}
In this paper, we investigate the existence of Sierpi\'{n}ski numbers and Riesel numbers as binomial coefficients. We show that for any odd positive integer $r$, there exist infinitely many Sierpi\'{n}ski numbers and Riesel numbers of the form $\binom{k}{r}$.  Let $S(x)$ be the number of positive integers $r$ satisfying $1\leq r\leq x$ for which $\binom{k}{r}$ is a Sierpi\'{n}ski number for infinitely many $k$.  We further show that the value $S(x)/x$ gets arbitrarily close to 1 as $x$ tends to infinity.  Generalizations to base $a$-Sierpi\'{n}ski numbers and base $a$-Riesel numbers are also considered. In particular, we prove that there exist infinitely many positive integers $r$ such that $\binom{k}{r}$ is simultaneously a base $a$-Sierpi\'{n}ski and base $a$-Riesel number for infinitely many $k$.
\end{abstract}

\section{Introduction}
In 1956, Riesel showed that if $k\equiv 509203\pmod{1184810}$, then for any natural number $n$, the value $k\cdot 2^n-1$ is composite \cite{riesel}.  Today we say that $k$ is a Riesel number if $k$ is an odd positive integer such that $k\cdot 2^n-1$ is composite for all natural numbers $n$.  Using methods similar to Riesel, Sierpi\'{n}ski showed in 1960 that there are infinitely many odd positive integers $k$ such that $k\cdot 2^n+1$ is composite for all natural numbers $n$ \cite{sierpinski}; values of $k$ satisfying this property are now known as Sierpi\'{n}ski numbers.  

In 2003, Chen showed that if $r\not\equiv 0,4,6,8\pmod{12}$, then there exist infinitely many odd positive integers $k$ such that $k^r$ is a Sierpi\'{n}ski number \cite{chen}.  Chen's result was later extended by Filaseta, Finch, and Kozek for all positive integers $r$ \cite{ffk}.  In their article, Filaseta, Finch, and Kozek asked the following question.

\begin{question}\label{question:ffk}
Let $f\in\mathbb{Z}[x]$.  Does there exist an integer $k$ such that $f(k)$ is a Sierpi\'{n}ski number?
\end{question}

This question has been studied by various authors.  For example, Finch, Harrington, and Jones studied this question for $f(x)\in\{x^r+x+c, ax^r+c, x^r+1, x^r+x+1\}$ \cite{fhj} and Emadian, Finch-Smith, and Kallus studied this question for $f(x)=384x^3+432x^2+112x-5$ \cite{efk}.  Other authors considered Question~\ref{question:ffk} for polynomials $f\in\mathbb{Q}[x]$.  Of particular note is the existence of infinitely many Sierpi\'{n}ski numbers in the sequence of triangular numbers and other polygonal numbers.  Recall that for $s\geq 3$, the $x$-th $s$-gonal number is given by 
$$P_s(x)=\frac{s-2}{2}x^2-\frac{s-4}{2}x.$$
Question~\ref{question:ffk} with respect to $P_s(x)$ has been studied by Baczkowski \emph{et al.} \cite{befks} and Baczkowski and Eitner \cite{be}.

In this article, we study Question~\ref{question:ffk} with respect to the polynomial 
$$\binom{x}{r}=\frac{x(x-1)(x-2)\cdots(x-(r-1))}{r!}$$
where $r$ is a fixed positive integer. Notice that the case $\binom{x}{2}$ has been previously studied since $\binom{x}{2}=P_3(x-1)$.  Of course, $\binom{x}{r}$ is more commonly referred to as the \emph{binomial coefficient} function.  We begin our investigation on the existence of Sierpi\'{n}ski binomial coefficients for general $r$ in Section~\ref{sec:sierpinski}, and extend some of these results to base $a$-Sierpi\'{n}ski and $a$-Riesel binomial coefficients in Section~\ref{sec:generalization}.

\section{Preliminary results, definitions, and notation}
Throughout this article, we use $[a,b]$ to denote the set of integers $x$ such that $a\leq x\leq b$.

For our investigation, we will make use of the following concept, originally introduced by Erd\H{o}s.

\begin{definition}
A \emph{covering system} of the integers is a finite collection of congruences such that every integer satisfies at least one congruence from the set.
\end{definition}
In this article, we will primarily use covering systems of the form:
\begin{equation}\label{eq:covering}
    \begin{split}
    0\pmod{2^\tau}\qquad&\text{where }\tau\text{ is a positive integer}\\
    2^{\ell-1}\pmod{2^{\ell}}\qquad&\text{ for each }1\leq\ell\leq \tau.
\end{split}
\end{equation}

Many of the proofs in this article rely heavily on the following two theorems, originally due to Zsigmondy \cite{zsigmondy} and Lucas \cite{lucas}, respectively.

\begin{theorem}[Zsigmondy's Theorem]\label{thm:bang}
Let $a$ and $b$ be relatively prime positive integers with $a>b$.  Then for any integer $n\geq 2$, there exists a prime $p$ such that $p$ divides $a^n-b^n$ and $p$ does not divide $a^{\widetilde{n}}-b^{\widetilde{n}}$ for any $\widetilde{n}<n$, with the exceptions
\begin{itemize}
    \item $(a,b)=(2,1)$ and $n=6$; and
    \item $a+b$ is a power of $2$ and $n=2$.
\end{itemize}
\end{theorem}

\begin{theorem}[Lucas' Theorem]\label{thm:lucas}
Let $p$ be a prime, and let $m$ and $n$ be nonnegative integers. Let the base $p$ representations of $m$ and $n$ be $m=\sum_{i=0}^jm_ip^i$ and $n=\sum_{i=0}^jn_ip^i$, respectively, where $m_i,n_i\in[0,p-1]$ for all $i\in[0,j]$. Then
$$\binom{m}{n}\mod{\prod_{i=0}^j\binom{m_i}{n_i}}{p}.$$
\end{theorem}

\section{Sierpi\'{n}ski binomial coefficients}\label{sec:sierpinski}

\begin{lemma}\label{lem:1}
Let $p$ be a prime, and let $r$ be a nonnegative integer. Let $j$ be the smallest nonnegative integer such that $r<p^{j+1}$. Then for all positive integers $k$ such that $k\mod{r}{p^{j+1}}$, we have
$$\binom{k}{r}\mod{1}{p}.$$
\end{lemma}

\begin{proof}
Let the base $p$ representations of $r$ and $k$ be $r=\sum_{i=0}^{j'}r_ip^i$ and $k=\sum_{i=0}^{j'}k_ip^i$, respectively, where $j\leq j'$, $k_i=r_i\in[0,p-1]$ for all $i\in[0,j]$, $r_i=0$ for all $i\in[j+1,j']$, and $k_i\in[0,p-1]$ for all $i\in[j+1,j']$. By Theorem~\ref{thm:lucas},
$$\binom{k}{r}\equiv\left(\prod_{i=0}^{j}\binom{k_i}{r_i}\right)\left(\prod_{i=j+1}^{j'}\binom{k_i}{r_i}\right)\equiv\left(\prod_{i=0}^{j}\binom{r_i}{r_i}\right)\left(\prod_{i=j+1}^{j'}\binom{k_i}{0}\right)\mod{1}{p}.$$
\end{proof}

The following three lemmas are verified computationally by Mathematica. The code for these lemmas is included in Appendix~\ref{appendixA}, Appendix~\ref{appendixB}, and Appendix~\ref{appendixC}, respectively.

\begin{lemma}\label{lem:G}
Let $p=641$, and let
\begin{align*}\mathcal{G}&=\{\gamma\in[1,p-1]:\gamma\text{ is odd}\}\\
&\hspace{20pt}\cup\{2,6,8,10,12,22,24,30,32,34,44,46,48,52,56,66,70,74,80,84,86,94,100,102,\\
&\hspace{40pt}104,110,118,120,134,136,140,144,146,160,162,174,176,182,184,190,194,\\
&\hspace{40pt}198,200,202,208,222,224,236,248,250,252,260,270,292,294,304,312,318,\\
&\hspace{40pt}334,336,338,348,366,368,374,402,414,424,426,454,474,530,546,552,578\}.
\end{align*}
Then there exists a function $\kappa:\mathcal{G}\to[0,p-1]$ such that for every $r\in\mathcal{G}$,
$$\binom{\kappa(r)}{r}\mod{-1}{p}.$$
\end{lemma}

\begin{lemma}\label{lem:pair}
Let $p=641$. Recall $\mathcal{G}$ defined in Lemma~$\ref{lem:G}$. Then there exist a function $\widetilde{\kappa}=(\widetilde{\kappa}',\widetilde{\kappa}''):[1,515]^2\to[0,p-1]^2$ such that for every ordered pair $ (r',r'')\in[1,515]^2$,
$$\binom{\widetilde{\kappa}'(r',r'')}{r'}\binom{\widetilde{\kappa}''(r',r'')}{r''}\mod{-1}{p}.$$
\end{lemma}

\begin{lemma}\label{lem:10primes}
Let $\mathcal{P}$ be the following set of primes $p$ that divides $2^{2^{\tau-1}}+1$ for some $\tau\in\N$ such that $\left(2^{2^{\tau-1}}+1\right)/p$ is divisible by another prime distinct from $p$:
{\small$$\{641,114689,274177,319489,974849,2424833,6700417,13631489,26017793,45592577,63766529\}.$$}
Then for every $r\in[1,640]$, there exists $p\in\mathcal{P}$ and $k\in\N$ such that
$$\binom{k}{r}\mod{-1}{p}.$$
\end{lemma}

\begin{lemma}\label{lem:-1}
Let $p=641$. Recall $\mathcal{G}$ and $\kappa$ defined in Lemma~$\ref{lem:G}$, and recall $\widetilde{\kappa}=(\widetilde{\kappa}',\widetilde{\kappa}'')$ defined in Lemma~$\ref{lem:pair}$. Let $r$ be a nonnegative integer with base $p$ representation $r=\sum_{i=0}^jr_ip^i$, where $r_i\in[0,p-1]$ for all $i\in[0,j]$.
\begin{enumerate}[\indent$(a)$]
\item\label{item:onedigit} If there exists $i_0\in[0,j]$ such that $r_{i_0}\in\mathcal{G}$, then for all positive integers $k$ such that $k\mod{r+(\kappa(r_{i_0})-r_{i_0})p^{i_0}}{p^{j+1}}$, we have
$$\binom{k}{r}\mod{-1}{p}.$$
\item\label{item:twodigits} If there exist $i_1,i_2\in[0,j]$ such that $r_{i_1},r_{i_2}\in[1,515]$, then for all positive integers $k$ such that $k\mod{r+(\widetilde{\kappa}'(r_{i_1},r_{i_2})-r_{i_1})p^{i_1}+(\widetilde{\kappa}''(r_{i_1},r_{i_2})-r_{i_2})p^{i_2}}{p^{j+1}}$, we have
$$\binom{k}{r}\mod{-1}{p}.$$
\end{enumerate}
\end{lemma}

\begin{proof}
$(\ref{item:onedigit})$ Let the base $p$ representation of $k$ be $k=\sum_{i=0}^{j'}k_ip^i$, where $j\leq j'$, $k_i=r_i$ for all $i\in[0,j]\setminus\{i_0\}$, $k_{i_0}=\kappa(r_{i_0})$, and $k_i\in[0,p-1]$ for all $i\in[j+1,j']$. Furthermore, define $r_i=0$ for all $i\in[j+1,j']$. By Theorem~\ref{thm:lucas},
$$\binom{k}{r}\equiv\left(\prod_{i=0}^{j}\binom{k_i}{r_i}\right)\left(\prod_{i=j+1}^{j'}\binom{k_i}{r_i}\right)\equiv\left(\prod_{\substack{i=0\\i\neq i_0}}^{j}\binom{r_i}{r_i}\right)
\binom{\kappa(r_{i_0})}{r_{i_0}}\left(\prod_{i=j+1}^{j'}\binom{k_i}{0}\right)\mod{-1}{p}.$$

$(\ref{item:twodigits})$ Let the base $p$ representation of $k$ be $k=\sum_{i=0}^{j'}k_ip^i$, where $j\leq j'$, $k_i=r_i$ for all $i\in[0,j]\setminus\{i_1,i_2\}$, $k_{i_1}=\widetilde{\kappa}'(r_{i_1},r_{i_2})$, $k_{i_2}=\widetilde{\kappa}''(r_{i_1},r_{i_2})$, and $k_i\in[0,p-1]$ for all $i\in[j+1,j']$. Furthermore, define $r_i=0$ for all $i\in[j+1,j']$. By Theorem~\ref{thm:lucas},
$$\binom{k}{r}\equiv\left(\prod_{\substack{i=0\\i\notin\{i_1,i_2\}}}^{j}\binom{r_i}{r_i}\right)
\binom{\widetilde{\kappa}'(r_{i_1},r_{i_2})}{r_{i_1}}\binom{\widetilde{\kappa}''(r_{i_1},r_{i_2})}{r_{i_2}}\left(\prod_{i=j+1}^{j'}\binom{k_i}{0}\right)\mod{-1}{p}.$$
\end{proof}

\begin{theorem}\label{thm:sierpinski}
Let $p=641$, and recall $\mathcal{G}$ defined in Lemma~$\ref{lem:G}$. Let $r$ be a nonnegative integer with base $p$ representation $r=\sum_{i=0}^jr_ip^i$, where $r_i\in[0,p-1]$ for all $i\in[0,j]$, such that at least one of the following conditions is satisfied:
\begin{enumerate}[\indent$(i)$]
\item\label{item:sierpinskiG} there exists $i_0\in[0,j]$ such that $r_{i_0}\in\mathcal{G}$; or
\item\label{item:sierpinskipairs} there exists $i_1,i_2\in[0,j]$ such that $r_{i_1},r_{i_2}\in[1,515]$.
\end{enumerate}
Then there exist infinitely many positive integers $k$ such that $\binom{k}{r}$ is a Sierpi\'{n}ski number.
\end{theorem}

\begin{proof}
Let $p_0=641$, $p_1=3$, $p_2=5$, $p_3=17$, $p_4=257$, $p_5=65537$, and $p_6=6700417$. Note that for each $\ell\in[1,6]$,
\begin{center}
$p_\ell\mid2^{2^\ell}-1$ and $p_\ell\nmid2^{2^{\widetilde{\ell}}}-1$ for any $\widetilde{\ell}<\ell$,
\end{center}
so we also have $2^{2^{\ell-1}}\mod{-1}{p_\ell}$.

Consider the covering system \eqref{eq:covering} with $\tau=6$. Suppose that $n\mod{2^{\ell-1}}{2^\ell}$ for some $\ell\in[1,6]$. Then
$$2^n=\left(2^{2^\ell}\right)^t\cdot2^{2^{\ell-1}}\equiv1^t\cdot(-1)\mod{-1}{p_\ell}$$
for some nonnegative integer $t$. Hence,
$$\binom{k}{r}\cdot2^n+1\mod{-\binom{k}{r}+1}{p_\ell}.$$
Let $j_\ell$ be the smallest nonnegative integer such that $r<p_\ell^{j_\ell+1}$ for each $\ell\in[1,6]$. By Lemma~\ref{lem:1}, if 
\begin{equation}\label{eq:kconditions1}
k\equiv r\pmod{p_\ell^{j_\ell+1}},
\end{equation}
then $\binom{k}{r}\cdot 2^n+1\equiv 0\pmod{p_\ell}$.

Since \eqref{eq:covering} is a covering system, if $n\not\mod{2^{\ell-1}}{2^\ell}$ for any $\ell\in[1,6]$, then $n\mod{0}{2^6}$. Note that $p_0\mid2^{2^6}-1$, so $2^n\equiv 1\pmod{p_0}$ and 
$$\binom{k}{r}\cdot 2^n+1\equiv\binom{k}{r}+1\pmod{p_0}.$$
Let $j_0$ be the smallest nonnegative integer such that $r<p_0^{j_0+1}$. Recall the function $\kappa$ defined in Lemma~\ref{lem:G}. By Lemma~$\ref{lem:-1}(\ref{item:onedigit})$, if condition~$(\ref{item:sierpinskiG})$ of this theorem is satisfied and 
\begin{equation}\label{eq:kconditions2}
k\mod{r+(\kappa(r_{i_0})-r_{i_0})p_0^{i_0}}{p_0^{j_0+1}},
\end{equation}
then $\binom{k}{r}\cdot 2^n+1\mod{0}{p_0}$.

Hence, for any natural number $n$, if the congruence in \eqref{eq:kconditions1} is satisfied for each $\ell\in[1,6]$ and the congruence in \eqref{eq:kconditions2} is satisfied, then $\binom{k}{r}\cdot 2^n+1$ is divisible by some prime $p_\ell$ with $0\leq\ell\leq 6$.  Using Lemma~\ref{lem:1}, we ensure that $\binom{k}{r}$ is odd by further requiring $k\equiv r\pmod{2^{j+1}}$, where $j$ is the smallest nonnegative integer such that $r<2^{j+1}$.  By the Chinese remainder theorem, there are infinitely many such integers $k$.  Choosing $k$ so that $\binom{k}{r}\geq p_6$ ensures that $\binom{k}{r}$ is a Sierpi\'{n}ski number.

If condition~$(\ref{item:sierpinskipairs})$ of this theorem is satisfied, then the same argument applies by replacing Lemma~$\ref{lem:-1}(\ref{item:onedigit})$ and \eqref{eq:kconditions2} with Lemma~$\ref{lem:-1}(\ref{item:twodigits})$ and the congruence
$$k\mod{r+(\widetilde{\kappa}'(r_{i_1},r_{i_2})-r_{i_1})p_0^{i_1}+(\widetilde{\kappa}''(r_{i_1},r_{i_2})-r_{i_2})p_0^{i_2}}{p_0^{j_0+1}}.$$ 
\end{proof}

The following corollary follows from Theorem~$\ref{thm:sierpinski}(\ref{item:sierpinskiG})$ since every odd positive integer must have an odd digit in its base $p$ representation.

\begin{corollary}\label{cor:corollary}
Let $r$ be an odd positive integer. Then there exist infinitely many positive integers $k$ such that $\binom{k}{r}$ is a Sierpi\'{n}ski number.
\end{corollary}

There are $245$ integers $r\in[1,2563]$ that do not satisfy the conditions in Theorem~\ref{thm:sierpinski}. Nonetheless, we can tackle these values of $r$ in the following theorem.

\begin{theorem}\label{thm:sierpinski10primes}
Let $r\in[1,2563]$. Then there exist infinitely many positive integers $k$ such that $\binom{k}{r}$ is a Sierpi\'{n}ski number.
\end{theorem}
\begin{proof}
If $r\in[641,2563]$, then the conclusion follows from Theorem~$\ref{thm:sierpinski}(\ref{item:sierpinskiG})$ since the base $p$ representation of $r$ contains the digits $1$, $2$, or $3$, which are in $\mathcal{G}$ defined in Lemma~\ref{lem:G}.

Suppose that $r\in[1,640]$. Let $\mathcal{P}$ be the set of primes defined in Lemma~\ref{lem:10primes}.  By Lemma~\ref{lem:10primes}, there exist $p_0\in\mathcal{P}$ and $k'\in\mathbb{N}$ such that $\binom{k'}{r}\mod{-1}{p_0}$. By the definition of $\mathcal{P}$, there is some integer $\tau\geq 5$ and some prime $p_\tau\neq p_0$ such that $p_0$ and $p_\tau$ both divide $2^{2^{\tau-1}}+1$. Consequently, $p_0$ and $p_\tau$ are both prime factors of $2^{2^{\tau}}-1$. By Theorem~\ref{thm:bang}, for each $\ell\in[1,\tau-1]$, let $p_\ell$ be a prime such that
\begin{center}
$p_\ell\mid2^{2^\ell}-1$ and $p_\ell\nmid2^{2^{\widetilde{\ell}}}-1$ for any $\widetilde{\ell}<\ell$,
\end{center}
so we also have $2^{2^{\ell-1}}\mod{-1}{p_\ell}$. Note that $p_0$ and $p_\tau$ are distinct from $p_\ell$ for all $\ell\in[1,\tau-1]$. This is because $2^{2^\ell}\mod{1}{p_\ell}$, implying that $2^{2^{\tau-1}}\mod{1}{p_\ell}$, while $2^{2^{\tau-1}}\mod{-1}{p_0}$ and $2^{2^{\tau-1}}\mod{-1}{p_\tau}$.

Consider the covering system \eqref{eq:covering}. Suppose that $n\mod{2^{\ell-1}}{2^\ell}$ for some $\ell\in[1,\tau]$. Let $j_\ell$ be the smallest nonnegative integer such that $r<p^{j_\ell+1}$. Similar to the argument presented in proof of Theorem~\ref{thm:sierpinski}, by Lemma~\ref{lem:1}, if
\begin{equation}\label{eq:kcon1}
k\mod{r}{p_{\ell}^{j_{\ell}+1}},
\end{equation} 
then $\binom{k}{r}\cdot 2^{n}+1\mod{0}{p_\ell}$.

Since \eqref{eq:covering} is a covering system, if $n\not\mod{2^{\ell-1}}{2^\ell}$ for any $\ell\in[1,\tau]$, then $n\mod{0}{2^\tau}$. Note that $r<p_0$, so by the definition of $k'$, for all $k\in\N$ such that 
\begin{equation}\label{eq:kcon2}
k\mod{k'}{p_0},
\end{equation} 
we have $\binom{k}{r}\mod{-1}{p_0}$, which implies that $\binom{k}{r}\cdot 2^n+1\equiv 0\pmod{p_0}$.  

The result follows by letting $k\geq\max\{p_0,p_1,\ldots,p_\tau\}$ satisfy the congruence relations \eqref{eq:kcon1} for all $\ell\in[1,\tau]$, \eqref{eq:kcon2}, and $k\mod{r}{2^{j+1}}$, where $j$ is the smallest nonnegative integer such that $r<2^{j+1}$.
\end{proof}

There are $641^2-1=410880$ one-digit or two-digit positive integers $\overline{r'r''}$ in base $641$, and from the code given in Appendix~\ref{appendixB}, only $3771-1=3770$ of them do not have any solution $(x',x'')\in[0,640]^2$ for the equation
$$\binom{x'}{r'}\binom{x''}{r''}\mod{-1}{641}.$$ 
For a positive integer $x$, let $S(x)$ be the number of $r\in[1,x]$ such that $\binom{k}{r}$ is a Sierpi\'{n}ski number for infinitely many positive integers $k$. Then $S(410880)/410880>99\%$, and the next theorem addresses $S(x)/x$ as $x$ tends to infinity.

\begin{theorem}\label{thm:sdensity}
The density $S(x)/x$ gets arbitrarily close to $1$ as $x$ tends to infinity.
\end{theorem}

\begin{proof}
Let $p=641$. Note that the cardinality of $\mathcal{G}$, which is defined in Lemma~\ref{lem:G}, is $395$. Hence, the number of integers less than $p^{j+1}$  such that no digit comes from $\mathcal{G}$ when expressed in base $p$ is
$$1-\frac{S(p^{j+1}-1)}{p^{j+1}-1}\leq\frac{(p-395)^{j+1}-1}{p^{j+1}-1},$$
which tends to $0$ as $j$ tends to infinity.
\end{proof}

\section{Generalizations of Sierpi\'{n}ski and Riesel binomial coefficients}\label{sec:generalization}

In 2009, Brunner \emph{et al.} generalized the concept of a Sierpi\'{n}ski number in the following way \cite{bckl}.

\begin{definition}
For a positive integer $a$, we call a positive integer $k$ an \emph{$a$-Sierpi\'{n}ski number} if $\gcd(k+1,a-1)=1$, $k$ is not a power of $a$, and $k\cdot a^n+1$ is composite for all natural numbers $n$.
\end{definition}

The following is an analogous definition for an $a$-Riesel number.
\begin{definition}
For a positive integer $a$, we call a positive integer $k$ an \emph{$a$-Riesel number} if $\gcd(k-1,a-1)=1$, $k$ is not a power of $a$, and $k\cdot a^n-1$ is composite for all natural numbers $n$.
\end{definition}

The next theorem is a generalization of Corollary~\ref{cor:corollary}.

\begin{theorem}\label{thm:oddr}
Let $a$ and $r$ be positive integers such that $a+1$ is not a power of $2$ and $r$ is odd.  Further assume that there exists a positive integer $\tau$ such that $a^{2^\tau}-1$ is divisible by distinct primes $p_0$ and $p_\tau$, where neither $p_0$ nor $p_\tau$ divides $a^{2^{\widetilde{\ell}}}-1$ for any $\widetilde{\ell}\in[0,\tau-1]$.  Then each of the following holds:
\begin{enumerate}[\indent$(a)$]
\item there exist infinitely many positive integers $k$ such that $\binom{k}{r}$ is an $a$-Sierpi\'{n}ski number;
\item there exist infinitely many positive integers $k$ such that $\binom{k}{r}$ is an $a$-Riesel number.
\end{enumerate}
\end{theorem}
\begin{proof}
For each $\ell\in[1,\tau]$, let $p_\ell$ be a prime such that 
\begin{center}
$p_\ell\mid a^{2^\ell}-1$ and $p_\ell\nmid a^{2^{\widetilde{\ell}}}-1$ for any $\widetilde{\ell}\in[0,\ell-1]$,
\end{center}
so we also have $a^{2^{\ell-1}}\equiv -1\pmod{p_\ell}$.  Note that such primes exist by Theorem~\ref{thm:bang}.  Let $p_{\tau+1},p_{\tau+2},\dotsc,p_\sigma$ be all the prime factors of $a-1$. Further let $p_{\sigma+1}$ be a prime factor of $a$. Note that $p_\ell$ are all distinct for $\ell\in[0,\sigma+1]$ since $\gcd(a,a^{\widetilde{\ell}}-1)=1$ for all positive integers $\widetilde{\ell}$.  For each $\ell\in[0,\sigma+1]$, let $j_\ell$ be the smallest positive integer satisfying $r<p_\ell^{j_\ell+1}$.

Using the Chinese remainder theorem, let $k$ satisfy the following congruences:
\begin{equation}\label{eq:3gen}
\begin{aligned}
k&\equiv 0\pmod{p_\ell^{j_\ell}}\text{ for each }\ell\in[\tau+1,\sigma]\text{ and}\\
k&\equiv r\pmod{p_{\sigma+1}^{j_{\sigma+1}+1}}.
\end{aligned}
\end{equation}
It follows from Theorem~\ref{thm:lucas} that $\binom{k}{r}\equiv 0\pmod{p_\ell}$ for each $\ell\in[\tau+1,\sigma]$ and $\binom{k}{r}\equiv 1\pmod{p_{\sigma+1}}$.  Consequently, $\gcd\left(\binom{k}{r}-1,a-1\right)=\gcd\left(\binom{k}{r}+1,a-1\right)=1$ and $\binom{k}{r}$ is not a power of $a$.  

For each $\ell\in[0,\tau]$, if 
\begin{equation}\label{eq:2gen}
k\equiv r\pmod{p_\ell^{j_\ell+1}},
\end{equation}
then $\binom{k}{r}\equiv 1\pmod{p_\ell}$ by Lemma~\ref{lem:1}.  Let $\sum_{i=0}^{j_{\ell}}r_{\ell i}p_\ell^i$ be the base $p_\ell$ representation of $r$.  Since $r$ is an odd integer, there exists an $i_0\in[0,j_\ell]$ such that $r_{\ell i_0}$ is odd.  By Theorem~\ref{thm:lucas}, if 
\begin{equation}\label{eq:1gen}
k\equiv r+(p_\ell-1-r_{\ell i_0})p_\ell^{i_0}\pmod{p_\ell^{j_\ell+1}},
\end{equation}
then $\binom{k}{r}\equiv\binom{p_\ell-1}{r_{\ell i_0}}\equiv -1\pmod{p_\ell}$.

Consider the covering system \eqref{eq:covering}.  If $n\equiv 2^{\ell-1}\pmod{2^\ell}$ for some $\ell\in[1,\tau]$, then $a^n\equiv-1\pmod{p_\ell}$, and if $n\equiv 0\pmod{p_0}$, then $a^n\equiv 1\pmod{p_0}$.  Thus, using the Chinese remainder theorem to choose $k$ so that
\begin{itemize}
\item $\binom{k}{r}\geq\max\{p_0,p_1,\ldots,p_\tau\}$;
\item $k$ satisfies \eqref{eq:2gen} for each $\ell\in[1,\tau]$; and
\item $k$ satisfies \eqref{eq:1gen} when $\ell=0$,
\end{itemize}
we ensure that for any natural number $n$, $\binom{k}{r}a^n+1$ is composite and divisible by $p_\ell$ for some $\ell\in[0,\tau]$.  Similarly, using the Chinese remainder theorem to choose $k$ so that
\begin{itemize}
\item $\binom{k}{r}\geq\max\{p_0,p_1,\ldots,p_\tau\}$; \item $k$ satisfies \eqref{eq:2gen} when $\ell=0$; and
\item $k$ satisfies \eqref{eq:1gen} for each $\ell\in[1,\tau]$,
\end{itemize}
we ensure that for any natural number $n$, $\binom{k}{r}a^n-1$ is composite and divisible by $p_\ell$ for some $\ell\in[0,\tau]$.  Thus, the proof is finished by recalling that $k$ satisfies the congruences in \eqref{eq:3gen}.
\end{proof}

For a positive integer $x$, let $R(x)$ be the number of $r\in[1,x]$ such that $\binom{k}{r}$ is a Riesel number for infinitely many positive integers $k$.  The following theorem follows similarly to Theorem~\ref{thm:sdensity}.
\begin{theorem}
The density $R(x)/x$ gets arbitrarily close to 1 as $x$ tends to infinity.
\end{theorem}

In 2001, Chen introduced the concept of a $(2,1)$-primitive $m$-covering \cite{chen2}.  This concept was extended to the following definition by Harrington in 2015 \cite{harrington}.
\begin{definition}\label{def:ab}
A covering system $\mathcal{C}=\{q_\ell\pmod{m_\ell}\}_{\ell=1}^\tau$ is called an \emph{$(a,b)$-primitive $m$-covering} if every integer satisfies at least $m$ congruences of $\mathcal{C}$ and there exist distinct primes $p_1,p_2,\ldots,p_\tau$ such that for each $\ell\in[1,\tau]$,
\begin{center}
$p_\ell\mid a^{m_\ell}-b^{m_\ell}$ and $p_\ell\nmid a^{\widetilde{\ell}}-b^{\widetilde{\ell}}$ for any $\widetilde{\ell}<m_\ell$.
\end{center}
Furthermore, a covering system $\mathcal{C}$ is called an \emph{$(a,b)$-primitive disjoint $m$-covering} if $\mathcal{C}$ is an $(a,b)$-primitive $m$-covering that can be partitioned into $m$ disjoint $(a,b)$-primitive $1$-covering systems.
\end{definition}

Harrington showed that if $a$ and $b$ are relatively prime integers such that $a+b$ is not a power of $2$, then there exists an $(a,b)$-primitive disjoint 3-covering \cite{harrington}.  Thus, the following theorem provides immediate results when $m=3$.

\begin{theorem}\label{thm:infr}
Let $a$ be a positive integer for which there exists an $(a,1)$-primitive $m$-covering $\mathcal{C}$.  Then there exist infinitely many positive integers $r$ for which each of the following holds:
\begin{enumerate}[\indent$(a)$]
\item\label{item:a1sierprinski} there exist infinitely many positive integers $k$ such that $\gcd\left(\binom{k}{r}+1,a-1\right)=1$, $\binom{k}{r}$ is not a power of $a$, and $\binom{k}{r}\cdot a^n+1$ has at least $m$ distinct prime divisors for all natural numbers $n$;
\item\label{item:a1riesel} there exist infinitely many positive integers $k$ such that $\gcd\left(\binom{k}{r}-1,a-1\right)=1$, $\binom{k}{r}$ is not a power of $a$, and $\binom{k}{r}\cdot a^n-1$ has at least $m$ distinct prime divisors for all natural numbers $n$; and
\item\label{item:a1sierprinskiriesel} if $\mathcal{C}$ is an $(a,1)$-primitive disjoint $m$-covering, then there exist infinitely many positive integers $k$ such that $\gcd\left(\binom{k}{r}+1,a-1\right)=\gcd\left(\binom{k}{r}-1,a-1\right)=1$, $\binom{k}{r}$ is not a power of $a$, $\binom{k}{r}\cdot a^n+1$ and $\binom{k}{r}\cdot a^n-1$ are composite, and each of $\binom{k}{r}\cdot a^n+1$ and $\binom{k}{r}\cdot a^n-1$ has at least $\lfloor m/2\rfloor$ distinct prime divisors for all natural numbers $n$.
\end{enumerate}
\end{theorem}

\begin{proof}
Let $\mathcal{C}=\{q_\ell\pmod{m_\ell}\}_{\ell=1}^\tau$ be an $(a,1)$-primitive $m$ covering with distinct primes $p_1,p_2,\dotsc,p_\tau$ given by Definition~\ref{def:ab}. Let $p_{\tau+1},p_{\tau+2},\dotsc,p_\sigma$ be all the prime factors of $a-1$. Further let $p_{\sigma+1}$ be a prime factor of $a$. Note that $p_\ell$ are all distinct for $\ell\in[1,\sigma+1]$ due to Definition~\ref{def:ab} and that $\gcd(a,a^{\widetilde{\ell}}-1)=1$ for all positive integers $\widetilde{\ell}$.

$(\ref{item:a1sierprinski})$ By the Chinese remainder theorem, there exists a positive integer $R$ such that \begin{equation}\label{eq:a1sierprinskiR}
R\equiv\begin{cases}
a^{-q_\ell}\pmod{p_\ell}&\text{for all }\ell\in[1,\tau];\\
0\pmod{p_\ell}&\text{for all }\ell\in[\tau+1,\sigma];\\
1\pmod{p_{\sigma+1}}.
\end{cases}
\end{equation}
Let $J_1$ be the smallest nonnegative integer such that $R<p_\ell^{J_1+1}$ for all $\ell\in[1,\sigma+1]$. Again by the Chinese remainder theorem, there exist infinitely many positive integers $r>R$ such that $r\mod{1}{p_\ell^{J_1+1}}$ for all $\ell\in[1,\sigma+1]$. For each such $r$, let $J_2$ be the smallest nonnegative integer such that $r<p_\ell^{J_2+1}$ for all $\ell\in[1,\sigma+1]$. Once again by the Chinese remainder theorem, there exist infinitely many positive integers $k>r$ such that $k\mod{r+R-1}{p_\ell^{J_2+1}}$ for all $\ell\in[1,\sigma+1]$. For each such $k$, let $J_3$ be the smallest nonnegative integer such that $k<p_\ell^{J_3+1}$ for all $\ell\in[1,\sigma+1]$. For each $\ell\in[1,\sigma+1]$, let the base $p_\ell$ representations of $R$, $r$, and $k$ be $R=\sum_{i=0}^{J_1}R_{\ell i}p_\ell^i$, $r=1+\sum_{i=J_1+1}^{J_2}r_{\ell i}p_\ell^i$, and $k=\sum_{i=0}^{J_1}R_{\ell i}p_\ell^i+\sum_{i=J_1+1}^{J_2}r_{\ell i}p_\ell^i+\sum_{i=J_2+1}^{J_3}k_{\ell i}p_\ell^i$, respectively. By Theorem~\ref{thm:lucas},
$$\binom{k}{r}\equiv\binom{R_{\ell0}}{1}\left(\prod_{i=1}^{J_1}\binom{R_{\ell i}}{0}\right)\left(\prod_{i=J_1+1}^{J_2}\binom{r_{\ell i}}{r_{\ell i}}\right)\left(\prod_{i=J_2+1}^{J_3}\binom{k_{\ell i}}{0}\right)\equiv R_{\ell0}\mod{R}{p_\ell}.$$
Therefore, $\gcd\left(\binom{k}{r}+1,a-1\right)=1$ since $\binom{k}{r}+1\mod{1}{p_\ell}$ for all $\ell\in[\tau+1,\sigma]$, and $\binom{k}{r}$ is not a power of $a$ since $\binom{k}{r}\mod{1}{p_{\sigma+1}}$. Lastly, since $\mathcal{C}$ is an $(a,1)$-primitive $m$ covering, for each natural number $n$, there exist distinct $\ell_1,\ell_2,\dotsc,\ell_m\in[1,\tau]$ such that $n\mod{q_{\ell_\iota}}{m_{\ell_\iota}}$ for all $\iota\in[1,m]$. Thus, for each $\iota\in[1,m]$,
$$\binom{k}{r}\cdot a^n-1\equiv R\left((a^{m_{\ell_\iota}})^ta^{q_{\ell_\iota}}\right)-1\equiv a^{-q_{\ell_\iota}}a^{q_{\ell_\iota}}-1\mod{0}{p_{\ell_\iota}}$$
for some nonnegative integer $t$.

$(\ref{item:a1riesel})$ This proof resembles the proof of part $(\ref{item:a1sierprinski})$ after replacing \eqref{eq:a1sierprinskiR} by
$$R\equiv\begin{cases}
-a^{-q_\ell}\pmod{p_\ell}&\text{for all }\ell\in[1,\tau];\\
0\pmod{p_\ell}&\text{for all }\ell\in[\tau+1,\sigma];\\
1\pmod{p_{\sigma+1}}.
\end{cases}$$

$(\ref{item:a1sierprinskiriesel})$ Let $\mathcal{C}$ be partitioned into $\mathcal{C}_1,\mathcal{C}_2,\dotsc,\mathcal{C}_m$, where $\mathcal{C}_\lambda=\{q_{\lambda\ell}\pmod{m_{\lambda\ell}}\}_{\ell=1}^{\tau_\lambda}$ for each $\lambda\in[1,m]$, and $\tau_1+\tau_2+\dotsb+\tau_\lambda=\tau$. Let $\{p_{\lambda1},p_{\lambda2},\dotsc,p_{\lambda\tau_\lambda}:\lambda\in[1,m]\}$ be given by Definition~\ref{def:ab}. A similar proof as from part $(\ref{item:a1sierprinski})$ applies after replacing \eqref{eq:a1sierprinskiR} by
$$R\equiv\begin{cases}
a^{-q_{\lambda\ell}}\pmod{p_{\lambda\ell}}&\text{for all }\ell\in[1,\tau_\lambda],\text{ where }\lambda\in[1,\lfloor m/2\rfloor];\\
-a^{-q_{\lambda\ell}}\pmod{p_{\lambda\ell}}&\text{for all }\ell\in[1,\tau_\lambda],\text{ where }\lambda\in[\lceil m/2\rceil+1,m];\\
0\pmod{p_\ell}&\text{for all }\ell\in[\tau+1,\sigma];\\
1\pmod{p_{\sigma+1}}.
\end{cases}$$
\end{proof}

\section{Concluding remarks}
Theorem~\ref{thm:oddr} shows that for any integer $a\geq 2$ and any odd positive integer $r$, there are infinitely many $a$-Sierpi\'{n}ski numbers and infinitely many $a$-Riesel numbers of the form $\binom{k}{r}$.  Theorems~\ref{thm:sierpinski} and \ref{thm:sierpinski10primes} show that there are infinitely many Sierpi\'{n}ski numbers of the form $\binom{k}{r}$ for most even positive integers $r$; however, it is unknown if there are Sierpi\'{n}ski numbers of the form $\binom{k}{r}$ for an arbitrary even positive integer $r$.  Thus, we present the following conjecture.

\begin{conjecture}
For any positive integer $r$, there exist infinitely many positive integers $k$ for which $\binom{k}{r}$ is simultaneously a Sierpi\'{n}ski number and a Riesel number.
\end{conjecture}

We end this section with the following question regarding Catalan numbers.  Recall that the $k$-th Catalan number is $\frac{1}{k+1}\binom{2k}{k}$.

\begin{question}
Are there infinitely many Catalan numbers that are either Sierpi\'{n}ski numbers or Riesel numbers?
\end{question}

The constructions in this paper rely on fixing a positive integer $r$ prior to finding $k$ values for which $\binom{k}{r}$ is either Sierpi\'{n}ski or Riesel. Hence, a new technique might be required in order to tackle the existence of Sierpi\'{n}ski or Riesel Catalan numbers.

\section{Acknowledgments}
These results are based on work supported by the National Science Foundation under grant numbered DMS-1852378.

\appendix
\section{Appendix: Mathematica code for Lemma~\ref{lem:G}}\label{appendixA}

\texttt{p = 641;\\
good =  Complement[ Table[\\
\indent\indent If[ Or @@ Table[ Mod[Binomial[k, r], p] == p - 1, \{k, p - 1\}], r],\\
\indent \{r, 0, p - 1\}], \{Null\}]}\\

The output \texttt{good} is our desired set $\mathcal{G}$.\\

\section{Appendix: Mathematica code for Lemma~\ref{lem:pair}}\label{appendixB}

The variables \texttt{p} and \texttt{good} are defined in the code given in Appendix~\ref{appendixA}.\\
\\
\texttt{bad = Complement[ Table[r, \{r, 0, p - 1\}], good];\\
badbad = \{\};\\
Do[ If[ Not[ Or @@ Flatten[\\
\indent Table[ Mod[Binomial[k1, bad[[r1]]] * Binomial[k2, bad[[r2]]], p] == p - 1,\\
\indent\indent \{k1, p - 1\}, \{k2, p - 1\}]]],\\
\indent badbad = Append[badbad, \{bad[[r1]], bad[[r2]]\}]],\\
\indent \{r1, Length[bad]\}, \{r2, Length[bad]\}];\\
Or @@ Table[ 1 <= badbad[[i, 1]] <= 515 \&\& 1 <= badbad[[i, 2]] <= 515,\\
\indent \{i, Length[badbad]\}]}\\

The variable \texttt{badbad} contains all ordered pairs of $(r',r'')\in[0,640]^2$ that fail to satisfy our desired equation. If we want to further investigate by using \texttt{Length[badbad]}, the number of ordered pairs of $(r',r'')\in[0,640]^2$ that fail to satisfy our desired equation is $3771$. However, the final output is \texttt{False}, showing that there are no unordered pairs $\{r',r''\}\subseteq[1,515]$ that fails to satisfy our desired equation.\\

\section{Appendix: Mathematica code for Lemma~\ref{lem:10primes}}\label{appendixC}

\texttt{plist = \{641, 114689, 274177, 319489, 974849, 2424833, 6700417,  13631489,\\
\indent 26017793, 45592577, 63766529\};\\
And @@ Table[Or @@ Table[\\
\indent Solve[Product[k - j, \{j, 0, r - 1\}]/r! == p - 1, k, Modulus -> p] != \{\},\\
\indent\{p, plist\}], \{r, 640\}]}\\

The output is \texttt{True}, showing that every $r\in[1,640]$ satisfies our desired equation.


\begin{thebibliography}{99}

\bibitem{bckl} A.~Brunner, C.~Caldwell, D.~Krywaruczenko, and C.~Lownsdale, Generalizing Sierpi\'{n}ski numbers to base $b$, \emph{New Aspects of Analytic Number Theory, Proceedings of RIMS, Surikaisekikenkyusho Kokyuroku} (2009), 69--79.

\bibitem{befks} D.~Baczkowski, J.~Eitner, C.~Finch, M.~Kozek, and B.~Suminski, Polygonal, Sierpi\'{n}ski, and Riesel numbers, \emph{J.\ Integer Seq.\/} \textbf{18} (2015), Article 15.8.1.

\bibitem{be} D.~Baczkowski and J.~Eitner, Polyonal-Sierpi\'{n}ski-Riesel sequences with terms having at least two distinct prime divisors, \emph{INTEGERS} \textbf{16} (2016), Article A40.

\bibitem{chen2} Y.G.~Chen, On integers of the form $k-2^n$ and $k2^n+1$, \emph{J.\ Number Theory} \textbf{89} (2001), 121--125.

\bibitem{chen} Y.G.~Chen, On integers of the form $k^r-2^n$ and $k^r2^n+1$, \emph{J.\ Number Theory} \textbf{98} (2003), 310--319.

\bibitem{efk} E.~Emadian, C.~Finch-Smith, and M.~Kallus, Ruth-Aaron pairs containing Riesel or Sierpi\'{n}ski numbers, \emph{INTEGERS} \textbf{18} (2018), Article A72.

\bibitem{ffk} M.~Filaseta, C.~Finch, and M.~Kozek, On powers associated with Sierpi\'{n}ski numbers, Riesel numbers, and Polignac's conjecture, \emph{J.\ Number Theory} \textbf{128} (2008), 1916--1940.

\bibitem{fhj} C.~Finch, J.~Harrington, and L.~Jones, Nonlinear Sierpi\'{n}ski and Riesel numbers, \emph{J.\ Number Theory} \textbf{133} (2013), 534--544.

\bibitem{harrington} J.~Harrington, Two questions concerning coverings systems of the integers, \emph{Int.\ J.\ Number Theory} \textbf{11} (2015), 1739--1750.

\bibitem{lucas} E.~Lucas, Sur les congruences des nombres eul\'{e}riens et des coefficients diff\'{e}rentiels des fonctions trigonom\'{e}triques suivant un module premier, \emph{Bull.\ Soc.\ Math.\ France} \textbf{6} (1878), 49--54.

\bibitem{riesel} H.~Riesel, N\r{a}gra stora primal, \emph{Elementa} \textbf{39} (1956), 258--260.

\bibitem{sierpinski} W.~Sierpi\'{n}ski, Sure un probl\`eme concernant les nombres $k2^n+1$, \emph{Elem.\ Math.\/} \textbf{15} (1960), 73--74.

\bibitem{zsigmondy} K. Zsigmondy, Zur Theorie der Potenzreste, \emph{Monatsch.\ Math.\ Phys.\/} \textbf{3} (1892), 265--284.

\end{thebibliography}
\end{document}